\documentclass[12pt]{article}
\usepackage{amsmath,amssymb,amsthm}
\date{}

\newcommand{\R}{{\mathbb R}}
\newcommand{\N}{{\mathbb N}}
\newcommand{\D}{{\mathcal D}}
\newcommand{\F}{{\mathcal F}}
\newcommand{\const}{{\rm const}}
\newcommand{\supp}{\mathop{\rm supp}}
\renewcommand{\div}{\mathop{\rm div}\nolimits}
\newcommand{\codim}{\mathop{\rm codim}}
\renewcommand{\Im}{\mathop{\rm Im}}

\numberwithin{equation}{section}
\theoremstyle{plain}
\newtheorem{theorem}{Theorem}[section]
\newtheorem{lemma}{Lemma}[section]
\newtheorem{proposition}{Proposition}[section]
\newtheorem{corollary}{Corollary}[section]
\theoremstyle{definition}
\newtheorem{definition}{Definition}[section]
\newtheorem{remark}{Remark}[section]

\begin{document}
\title{On linear evolutionary equations with skew symmetric spatial operators}
\author{Evgeny Yu. Panov \\ St. Petersburg Department of V.\,A.~Steklov Institute \\ of Mathematics of the Russian Academy of Sciences, \\
St. Petersburg, Russia, \\ Yaroslav-the-Wise Novgorod State University,\\ Veliky Novgorod, Russia}
\maketitle

\begin{abstract}
We study generalized solutions of an evolutionary equation related to a densely defined skew-symmetric operator in a real Hilbert space. We establish existence of a contractive semigroup, which provides generalized solutions, and find criteria of uniqueness of generalized solutions. Some applications are given including the transport equations and the linearised Euler equations with  solenoidal (and generally discontinuous) coefficients. Under some additional regularity assumption on the coefficients we prove that the corresponding spatial operators are skew-adjoint, which implies existence and uniqueness of generalized solutions for both the forward and the backward Cauchy problem.
\end{abstract}

\section{Introduction}

Let $H$ be a real Hilbert space, $A_0$ be a skew-symmetric linear operator in $H$ with a dense domain $X_0=D(A_0)\subset H$. Skew-symmetricity of the operator $A_0$ means that $(A_0u,v)=-(u,A_0v)$ $\forall u,v\in X_0$, where $(\cdot,\cdot)$ denotes the scalar multiplication in $H$. The skew-symmetricity condition is equivalent to the relation
$-A_0\subset (A_0)^*$, where $(A_0)^*$ is the adjoint operator. Since an adjoint operator is closed, we conclude that there exists the closure $A$ of the operator $A_0$. This operator $A$ is a closed skew-symmetric operator, and $A^*=(A_0)^*$.
We underline that in view of the identity $2((A_0u,v)+(u,A_0v))=(A_0(u+v),u+v)-(A_0(u-v),u-v)$ the skew-symmetricity of $A_0$ is equivalent to the condition $(A_0u,u)=0$ $\forall u\in X_0$.

The operator $A^*$ is an extension of the operator $-A$, and may be not skew-symmetric. Its skew-symmericity holds only in the case $A^*=-A$, that is, when the operator $A$ is skew-adjoint.

We consider the evolutionary equation
\begin{equation}\label{e1}
u'-A^*u =0, \quad u=u(t), \ t\in\R_+,
\end{equation}
where $\R_+=[0,+\infty)$, with the initial condition
\begin{equation}\label{c1}
u(0)=u_0\in H.
\end{equation}

\begin{definition}\label{def1}
A function $u=u(t)\in L^\infty_{loc}(\R_+,H)$ is called a generalized solutions (g.s.) of problem (\ref{e1}), (\ref{c1}) if
$\forall f(t)\in C_0^1(\R_+,X_0)$, where the space $X_0$ equipped with the graph norm $\|u\|^2=\|u\|_H^2+\|A_0u\|_H^2$,
\begin{equation}\label{gs}
\int_{\R_+}(u(t),f'(t)+A_0f(t))dt+(u_0,f(0))=0.
\end{equation}
\end{definition}

\begin{remark}\label{rem1}
(1) It follows from relation (\ref{gs}) with $f(t)=v\varphi(t)$, where $\varphi(t)\in C_0^1((0,+\infty))$, $v\in X_0$, that
$\frac{d}{dt} (u(t),v)=(u(t),Av)$ in $\D'((0,+\infty))$, where we denote by $\D'(I)$ the space of distributions on an open set $I$. The latter relation directly implies weak continuity of $u(t)$ (after possible correction on a set of null measure). As is easy to see, relation (\ref{gs}) also implies initial condition
(\ref{c1}) understood in the sense of weak convergence $u(t)\rightharpoonup u_0$ as $t\to 0+$;

(2) Let $X=D(A)$ be the domain of operator $A$, equipped with the graph norm $\|u\|^2=\|u\|_H^2+\|A(u)\|_H^2$, then $X$ is a Banach space in view of closedness of $A$ (it is even a Hilbert space with scalar multiplication $(u,v)+(Au,Av)$). Since the operator $A$ is the closure of $A_0$, the space $X_0$ is dense both in $H$ and in $X$, which implies that the space $C_0^1(\R_+,X_0)$ is dense in $C_0^1(\R_+,H)\cap C_0(\R_+,X)$. Therefore, every test function $f(t)\in C_0^1(\R_+,H)\cap C_0(\R_+,X)$ can be approximated by a sequence
$f_r(t)\in C_0^1(\R_+,X_0)$, $r\in\N$, so that $f_r(t)\to f(t)$, $f_r'(t)\to f'(t)$, $A_0f_r(t)\to Af(t)$ as $r\to\infty$ in $H$, uniformly with respect to $t$. Passing to the limit as $r\to\infty$ in relations (\ref{gs}) with the test functions $f=f_r$, we arrive at the relation
\[
\int_{\R_+}(u(t),f'(t)+Af(t))dt+(u_0,f(0))=0
\]
for all $f(t)\in C_0^1(\R_+,H)\cap C_0(\R_+,X)$. In particular, since $C_0^1(\R_+,H)\cap C_0(\R_+,X)\supset C_0^1(\R_+,X)$, we can replace $A_0,X_0$ in Definition~\ref{def1} by  $A$ and $X$, respectively. Thus, we may initially suppose that $A_0=A$ is a closed skew-symmetric operator.
\end{remark}

\section{Contractive semigroups of g.s.}
We consider a $C_0$-semigroup $T_t=e^{tB}$ of bounded linear operators in $H$ with an infinitesimal generator $B$. It is known that $B$ is a densely defined closed operator.
We are interesting when this semigroup provides g.s. of the problem  (\ref{e1}), (\ref{c1}). The following theorem gives the answer.

\begin{theorem}\label{th1}
Functions $u(t)=e^{tB}u_0$ are g.s. of the problem (\ref{e1}), (\ref{c1}) for every $u_0\in H$ if and only if $B\subset A^*$.
\end{theorem}

\begin{proof}
First, we assume that $B\subset A^*$ and that $u_0\in D(B)$. Then the function $u(t)=e^{tB}u_0\in C(\R_+,D(B))\cap C^1(\R_+,H)$ and $u'(t)=T_tBu_0=BT_tu_0=Bu(t)$. Therefore, for each $f(t)\in C_0^1(\R_+,X_0)$ there exists a continuous derivative
\begin{equation}\label{sk1}
\frac{d}{dt}(u(t),f(t))=(u'(t),f(t))+(u(t),f'(t))=(Bu(t),f(t))+(u(t),f'(t)).
\end{equation}
By our assumption $B\subset A^*$ and therefore
\[
(Bu(t),f(t))=(A^*u(t),f(t))=(u(t),Af(t))=(u(t),A_0f(t)).
\]
Now it follows from (\ref{sk1}) that
\[
\frac{d}{dt}(u(t),f(t))=(u(t),A_0f(t)+f'(t)).
\]
After integration this relation implies
\[
\int_{\R_+}(u(t),f'(t)+A_0f(t))dt=\int_{\R_+}\frac{d}{dt}(u(t),f(t))dt=-(u_0,f(0)).
\]
Hence, identity (\ref{gs}) holds, and $u(t)$ is a g.s. of problem \ref{e1}), (\ref{c1}).

In the case of arbitrary $u_0\in H$ we consider a sequence
$u_{0k}\in D(B)$, which converges to $u_0$ in $H$ as $k\to\infty$. As was already established, the functions $u_k(t)=T_tu_{0k}$ are g.s. of problem (\ref{e1}), (\ref{c1}) with initial data $u_{0k}$ for all $k\in\N$, and
\[
\|u_k(t)-u(t)\|_H\le\|T_t\|\|u_{0k}-u_0\|_H\mathop{\to}_{k\to\infty} 0
\]
uniformly on any segment $[0,T]$, where we are taking into account the known for $C_0$-semigroups bound $\|T_t\|\le Ce^{\alpha t}$ with some constants $C>0$, $\alpha\in\R$.
Passing to the limit as $k\to\infty$ in the identity (\ref{gs})
\[
\int_{\R_+}(u_k(t),f'(t)+A_0f(t))dt+(u_{0k},f(0))=0,
\]
we conclude that for all test functions $f(t)\in C_0^1(\R_+,X_0)$
\[
\int_{\R_+}(u(t),f'(t)+A_0f(t))dt+(u_0,f(0))=0.
\]
Hence, $u(t)$ is a g.s. of problem (\ref{e1}), (\ref{c1}), as was to be proved.

Conversely, assume that the function $u(t)=T_tu_0$ is a g.s. of (\ref{e1}), (\ref{c1}) for each $u_0\in H$. If $u_0\in D(B)$, then  $u(t)=T_tu_0\in C^1(\R_+,H)$ and  $u'(0)=Bu_0$. Therefore, for all $v\in D(A_0)$
\[
I(t)=(u(t),v)\in C^1(\R_+), \ I'(0)=(Bu_0,v).
\]
On the other hand for all $h(t)\in C_0^1(\R_+)$
\[
\int_{\R_+} I(t)h'(t)dt=\int_{\R_+} (u(t),vh'(t))dt=
-h(0)(u_0,v)-\int_{\R_+} (u(t),A_0v)h(t)dt
\]
by identity (\ref{gs}) with $f=h(t)v$. Integrating by parts, we arrive at
\[
\int_{\R_+} I'(t)h(t)dt=\int_{\R_+} (u(t),A_0v)h(t)dt.
\]
Since $h(t)\in C_0^1(\R_+)$ is arbitrary, we obtain the equality $I'(t)=(u(t),A_0v)$. In particular,
$(u_0,A_0v)=I'(0+)=(Bu_0,v)$. The identity $(u_0,A_0v)=(Bu_0,v)$ holds for all $v\in D(A_0)$, and by the definition of the adjoint operator $A^*=(A_0)^*$ we obtain the inclusion $u_0\in D(A^*)$ and the equality $A^*u_0=Bu_0$. Thus, $B\subset A^*$, which completes the proof.
\end{proof}

By the Lumer-Phillips theorem (see, for instance, \cite[Theorem~1.1.3]{Phil}) a $C_0$-semigroup $T_t=e^{tB}$ is contractive (i.e., $\|T_t\|\le 1$ $\forall t>0$) if and only if its generator $B$ is an $m$-dissipative operator. Remind that an operator $A$ in a Banach space $H$ is called dissipative if $\|u-hAu\|\ge \|u\|$ $\forall u\in D(A)$, $h>0$. A dissipative operator $A$ is called $m$-dissipative if $\Im (E-hA)=X$ for all $h>0$ (it is sufficient that this property is satisfied only for one value $h>0$). It is known that an $m$-dissipative operator is a maximal dissipative operator, in the case of a Hilbert space $H$ the inverse statement is also true. Notice also that in the case of Hilbert space $H$ dissipativity of an operator $A$ reduces to the condition  $(Au,u)\le 0$ $\forall u\in D(A)$. In particular, skew-symmetric operators are dissipative. The following property means that the requirement $B\subset A^*$ of Theorem~\ref{th1} for an
$m$-dissipative operator is equivalent to the condition that $B$ is an extension of the operator $-A$.

\begin{lemma}\label{lem1}
Let $B$ be an $m$-dissipative operator in $H$. Then $B\subset A^*\Leftrightarrow -A\subset B$.
\end{lemma}

\begin{proof}
Let $B\subset A^*$. Since also $-A\subset A^*$, then $-Au=Bu=A^*u$ on $D(A)\cap D(B)$. This allows to construct an extension  $\tilde B$
of both the operators $-A$ and $B$ defined on $D(A)+D(B)$ by the equality
\[
\tilde Bu=-Au_1+Bu_2, \quad u=u_1+u_2, \ u_1\in D(A), u_2\in D(B).
\]
Since the operators $-A$ and $B$ coincide on their common domain $D(A)\cap D(B)$, the above definition of the operator $\tilde B$ is correct and does not depend on the representation $u=u_1+u_2$. By the construction $-A\subset \tilde B$, $B\subset \tilde B$. Let us show that the operator $\tilde B$ is dissipative. We take an arbitrary $u\in D(A)+D(B)$. Then $u=u_1+u_2$ for some $u_1\in D(A)$, $u_2\in D(B)$, and
\begin{align}\label{sk2}
(\tilde Bu,u)=(-Au_1+Bu_2,u_1+u_2)=\nonumber\\-(Au_1,u_1)+(Bu_2,u_2)+((u_1,Bu_2)-(Au_1,u_2))=\nonumber\\
(Bu_2,u_2)+((u_1,A^*u_2)-(Au_1,u_2))=(Bu_2,u_2)\le 0
\end{align}
by dissipativity of the operator $B$. We take also into account that $(Au_1,u_1)=0$ by skew-symmetricity of $A$ and that
$B\subset A^*$. By (\ref{sk2}) the operator $\tilde B$ is dissipative. But $B$ is a maximal dissipative operator and we conclude that $\tilde B=B$. In particular, $D(A)+D(B)=D(B)$, that is, $D(A)\subset D(B)$. This implies that $-A\subset B$.

Conversely, suppose that $-A\subset B$. Let $u\in D(A)$, $v\in D(B)$. Then $u+sv\in D(B)$ for all $s\in\R$ and by dissipativity of $B$ we have $f(s)\doteq (B(u+sv),u+sv)\le 0$. The function $f(s)$ is quadratic and can be written as
\begin{align}\label{sk3}
f(s)=(Bu,u)+s((Bu,v)+(Bv,u))+s^2(Bv,v)=\nonumber\\ s((Bu,v)+(Bv,u))+s^2(Bv,v),
\end{align}
because  $(Bu,u)=-(Au,u)=0$ by the condition $-A\subset B$. Notice also that $(Bv,v)\le 0$. It now follows from representation
(\ref{sk3}) and the condition $f(s)\le 0$ $\forall s\in\R$ that $(Bu,v)+(Bv,u)=0$. Thus, for all $v\in D(B)$ and $u\in D(A)$ \[(Au,v)=-(Bu,v)=(Bv,u)=(u,Bv).\]
This identity implies that $v\in D(A^*)$ and $A^*v=Bv$ for all $v\in D(B)$. This means that $B\subset A^*$. The proof is complete.
\end{proof}
We underline that the implication $-A\subset B \Rightarrow B\subset A^*$ was also established in \cite[Lemma~1.1.5]{Phil}.

\section{Main results}

By Theorem~\ref{th1} and Lemma~\ref{lem1} generators of contractive semigroups of g.s. to equation (\ref{e1}) are exactly $m$-dissipative extensions of the operator $-A$. Existence of such extensions is well-known, see for example
\cite{Phil}. We want to construct such an extension that is ``maximally close'' to a skew-symmetric operator. First, we need to choose a skew-symmetric extension of $-A$. Such extensions reduces to symmetric extensions of the symmetric operator $iA$, which are based on the Cayley transform. In the framework of skew-symmetric operators the Cayley transform $Q=(E+A)(E-A)^{-1}$ is even more natural. The operator $Q$ is an isometry between the spaces $H_-=\Im (E-A)$ and $H_+=\Im (E+A)$ (notice that the spaces $H_\pm$ are closed due to the closedness of operator $A$). The Cayley transform is invertible, the inverse transform is defined by the formula $A=(Q-E)(Q+E)^{-1}$ (under the assumption that the domain
$D(A)=\Im(Q+E)$ is dense, the operator $Q+E$ is always invertible). Hence, skew-symmetric extensions of the operator $A$ corresponds to isometric extensions of the operator $Q$, which are reduced to construction of isometric maps from $(H_-)^\perp$ into $(H_+)^\perp$. Remind that Hilbert dimensions of these spaces (i.e., Hilbert codimensions of $H_\pm$) are called the deficiency indexes of $A$. We denote them $d_+$, $d_-$, respectively, so that $d_\pm=d_\pm(A)=\codim H_\pm$. Maximal skew-symmetric operators $A$ are characterized by the condition that at least one of the deficiency indexes $d_\pm(A)$ is null. The condition $d_+=d_-=0$ describes skew-adjoint operators, their Cayley transforms $Q$ are orthogonal operators. It is known that any skew-symmetric operator can be extended to a maximal skew-symmetric operator. Let $\tilde A$ be a maximal skew-symmetric extension of $A$, and $d_\pm=d_\pm(\tilde A)$ be the deficiency indexes of this extension. Then one of these indexes is null. If $d_+=0$ then $\Im(E+\tilde A)=H$ and the operator $B=-\tilde A$ is $m$-dissipative.
Similarly, if $d_-=0$ the operator $\tilde A$ is $m$-dissipative. It is known (see, for instance, \cite[Theorem~1.1.2]{Phil}) that the operator $(\tilde A)^*$ is also $m$-dissipative, and that
$-A\subset -\tilde A\subset (\tilde A)^*$ by skew-symmetricity of $\tilde A$. Taking
$B=(\tilde A)^*$, we obtain an $m$-dissipative extension of $-A$ in the case $d_-=0$. We established the following result.

\begin{theorem}\label{th2}
There exists a contractive semigroup of g.s. $u(t)=e^{tB}u_0$ of equation (\ref{e1}) such that the following two cases are possible:

(i) The generator $B$ is skew-symmetric. In this case the operators $T_t=e^{tB}$ are isometric embeddings and the conservation of energy $\|u(t)\|=\|u_0\|$ holds for all $t>0$;

(ii) the adjoint operator $B^*$ is skew-symmetric. Then the operators $(T_t)^*$ are isometric embeddings.
\end{theorem}

Concerning the uniqueness of a semigroup of g.s., the following statement holds.

\begin{theorem}\label{th3}
A contractive semigroup of g.s. of equations (\ref{e1}) is unique if and only if the skew-symmetric operator $A$ is maximal.
\end{theorem}

\begin{proof}
Suppose that the skew-symmetric operator $A$ is maximal and that $d_\pm$ are its deficiency indexes. If $u(t)=e^{tB}u_0$ is a contractive semigroup of g.s. then its generator $B$ is $m$-dissipative and $-A\subset B\subset A^*$ by Theorem~\ref{th1} and Lemma~\ref{lem1}. As has been already demonstrated above, one of the operators $-A$, $A^*$ is $m$-dissipative, $-A$ in the case $d_+=0$ and $A^*$ if $d_-=0$. Since $B$ is an $m$-dissipative operator as well while such operators are maximal dissipative operators, we conclude that $B=-A$ if $d_+=0$, and $B=A^*$ if $d_-=0$ (notice that $B=-A=A^*$ in the case $d_+=d_-=0$). Thus, the operator $B$ is uniquely defined and the corresponding semigroup is also unique.

Conversely, if the operator $A$ is not maximal then both its deficiency indexes are not zero. As is easy to realize, then there exist two different maximal skew-symmetric extensions $\tilde A_1$, $\tilde A_2$ of the operator $A$ with the same deficiency  indexes $d_\pm$. The corresponding $m$-dissipative operators \[B_1=\left\{\begin{array}{lcr} -\tilde A_1 & , & d_+=0, \\ (\tilde A_1)^* & , & d_-=0, \end{array}\right. \quad B_2=\left\{\begin{array}{lcr} -\tilde A_2 & , & d_+=0, \\ (\tilde A_2)^* & , & d_-=0 \end{array}\right.\] are different and generate different contractive semigroups of g.s.
\end{proof}

Now we are going to prove uniqueness of g.s. to problem  (\ref{e1}), (\ref{c1}).

\begin{theorem}\label{th4} A g.s. of the Cauchy problem (\ref{e1}), (\ref{c1}) is unique if and only if the deficiency index $d_-$ of skew-symmetric operator $A$ equals zero.
\end{theorem}

\begin{proof}
Suppose that a g.s. of problem (\ref{e1}), (\ref{c1}) is unique. Let us show that the deficiency index $d_-$ of the operator $A$ is null. Assuming the contrary, $d_->0$, we find that the closed subspace $\Im (E-A)$ is a proper subspaces of $H$. Then $\ker(E-A^*)=(\Im(E-A))^\perp\not=\{0\}$ and therefore there exists a nonzero vector $u_0\in D(A^*)$ such that  $A^*u_0=u_0$. As is easy to see, the function $u=e^t u_0$ is a g.s. of problem (\ref{e1}), (\ref{c1}), different from the semigroup g.s. $T_tu_0$ constructed in Theorem~\ref{th2}, because the latter is bounded. We also notice that for all $t_0\ge 0$ the functions
\[
\tilde u(t)=\left\{\begin{array}{lcr} e^tu_0, & , & 0\le t\le t_0, \\ e^{t_0}T_{t-t_0}u_0 & , & t\ge t_0 \end{array} \right.
\]
are pairwise different bounded g.s. of the problem (\ref{e1}), (\ref{c1}), and we come to a contradiction even with the condition of uniqueness of a bounded (!) g.s. Hence, $d_-=0$.

Conversely, assume that $d_-=0$. Let $u(t)$ be a g.s. of (\ref{e1}), (\ref{c1}) with null initial data. By the linearity,
it is sufficient to prove that $u(t)\equiv 0$. Since $d_+(-A)=d_-=0$ then, according to the proof of Theorem~\ref{th2} the operator $A$ generates the semigroup  $u=e^{tA}u_0$ of g.s. to the equation $u'+A^*u=0$. Let $v_0\in D(A)$ and $v(t)=e^{(t_0-t)A}v_0$, $t\le t_0$. Then $v(t)\in C^1([0,t_0],H)\cap C([0,t_0],D(A))$ and $v'(t)=-Av(t)$.
We choose such a function $\beta(s)\in C_0(\R)$ that $\supp\beta(s)\subset [-1,0]$, $\beta(s)\ge 0$, $\displaystyle\int\beta(s)ds=1$, and set for ${\nu\in\N}$  \ $\beta_\nu(s)=\nu\beta(\nu s)$,
$\displaystyle\theta_\nu(t)=\int_t^{+\infty} \beta_\nu(s)ds$. It is clear that
\[\beta_\nu(s)\in C_0^1(\R), \ \supp\beta_\nu(s)\subset [-1/\nu,0], \ \beta_\nu(s)\ge 0, \ \int\beta_\nu(s)ds=1.\] Therefore, the sequence $\beta_\nu(s)$ converges as $\nu\to\infty$ to the Dirac
$\delta$-function in $\D'(\R)$ while $\theta_\nu(t)\in C^1(\R)$ are decreasing functions, $\theta_\nu(t)=1$ for $t\le -1/\nu$, $\theta_\nu(t)=0$ for $t\ge 0$ and the sequence $\theta_\nu(t)$ converges point-wise as $\nu\to\infty$ to the function $\theta(-t)$, where $\displaystyle\theta(s)=\left\{\begin{array}{ll} 0, & s\le 0, \\ 1, & s>0.
\end{array}\right. $ is the Heaviside function. The function $f(t)=v(t)\theta_\nu(t-t_0)$ lies in the space  $C_0^1(\R_+,H)\cap C_0(\R_+,D(A))$ (we agree that $f(t)=0$ for $t>t_0$), and by Remark~\ref{rem1}(2) it is a proper test function in relation (\ref{gs}) for the g.s. $u(t)$. Revealing this relation and taking into account the equality $f'(t)=v'(t)\theta_\nu(t-t_0)-v(t)\beta_\nu(t-t_0)=-Af(t)-v(t)\beta_\nu(t-t_0)$, we arrive at the equality
\[
-\int (u(t),v(t))\beta_\nu(t-t_0)dt=0.
\]
Since the function $(u(t),v(t))$ is continuous (this follows from strong continuity of $v(t)$ and weak continuity of $u(t)$, see Remark~\ref{rem1}(1)), this equality in the limit as $\nu\to\infty$ implies that $(u(t_0),v_0)=(u(t_0),v(t_0))=0$ for each $v_0\in D(A)$. By the density of $D(A)$ in $H$, we conclude that $u(t_0)=0$.  Since $t_0>0$ is arbitrary, $u(t)\equiv 0$, as was to be proved.
\end{proof}

We can also study the backward Cauchy problem for equation (\ref{e1}), considered for the time $t<T$ with the Cauchy data
$u(T,x)=u_0(x)$ at the final moment $t=T$. After the change $t\to T-t$ it is reduced to the standard problem (\ref{e1}), (\ref{c1}) but related to the operator $-A$ instead of $A$. Therefore, the uniqueness of g.s. to the backward Cauchy problem is equivalent to the condition $d_+=d_-(-A)=0$. In particular, the uniqueness of g.s. for both forward and backward
problems is equivalent to the requirement $d_+=d_-=0$, that is, to the condition that the operator $A$ is skew-adjoint.
We have proved the following statement.

\begin{corollary}\label{cor1} Uniqueness of g.s. to both the forward and the backward Cauchy problems for equation (\ref{e1}) is equivalent to skew-adjointness of the operator $A$.
\end{corollary}

We observe that for g.s. $u=\tilde u(t)$ constructed in the proof of Theorem~\ref{th4} the following energy inequality fails.
\begin{equation}\label{En}
\forall t>0 \quad E(t)\doteq\frac{1}{2}\|u(t)\|_2^2\le E(0)=\frac{1}{2}\|u_0\|_2^2.
\end{equation}
It is clear that for semigroup g.s. $u=e^{tB}u_0$ the energy $E(t)$ decreases and (\ref{En}) is fulfilled.
In the case of maximal skew-symmetric operator $A$ it turns out that any g.s. satisfying energy inequality coincides with the unique (by Theorem~\ref{th3}) semigroup g.s. More precisely, the following statement holds.

\begin{theorem}\label{th5}
A g.s. of the problem (\ref{e1}), (\ref{c1}), (\ref{En}) is unique for each initial data $u_0\in H$ if and only if the skew-symmetric operator $A$ is maximal.
\end{theorem}

\begin{proof}
By Theorem~\ref{th3} maximality of operator $A$ is necessary for uniqueness of a g.s. satisfying (\ref{En}).
To prove the sufficiency, assume that the skew-symmetric operator $A$ is maximal. Then at least one of its deficiency indexes $d_\pm$ is null. If $d_-=0$ then by Theorem~\ref{th4} uniqueness holds even in the wider class of whole g.s. and it only remains to treat the case $d_+=0$. In this case the operator $-A$ is $m$-dissipative and generates the contractive semigroup $u=e^{-tA}u_0$ of g.s. to problem (\ref{e1}), (\ref{c1}). Let $u(t)$ be a g.s. of problem (\ref{e1}), (\ref{c1}), (\ref{En}), $v\in D(A)$, $t_0>0$. $\nu\in\N$. We choose the test function $f(t)=\theta_\nu(t-t_0)v(t)$, where $v(t)=e^{-tA}v$ while the sequence $\theta_\nu=\int_t^{+\infty} \beta_\nu(s)ds$ was defined in the proof of Theorem~\ref{th4}. Then $f(t)\in C_0^1(\R_+,H)\cap C_0(\R_+,D(A))$, and $f'(t)+Af(t)=-v(t)\beta_\nu(t-t_0)$. It follows from relation (\ref{gs}) and Remark~\ref{rem1}(2) that for large enough $\nu\in\N$
\[
\int_{R_+} (u(t),v(t))\beta_\nu(t-t_0)dt=(u_0,v).
\]
This implies in the limit as $\nu\to\infty$ the equality $(u(t_0),v(t_0))=(u_0,v)$. Since $D(A)$ is dense in $H$ and  $v(t_0)=e^{-t_0A}v$ depends continuously on $v$, we can pass to the limit as $v\to u_0$ in the obtained equality. As a result, we arrive at the identity $(u(t),\tilde u(t))=\|u_0\|^2$ for all $t=t_0>0$ where $\tilde u(t)=e^{-tA}u_0$. Using the energy inequality (\ref{En}), we derive the relation
\[\|u_0\|^2=(u(t),\tilde u(t))\le\|u(t)\|\cdot\|\tilde u(t)\|\le\|u_0\|\cdot\|u_0\|=\|u_0\|^2,
\]
which implies the equalities
\[\|u(t)\|=\|\tilde u(t)\|=\|u_0\|, \quad (u(t),\tilde u(t))=\|u(t)\|\cdot\|\tilde u(t)\|. \]
Obviously, these equalities are possible only in the case when $u(t)=\tilde u(t)=e^{-tA}u_0$. Hence, a g.s. of (\ref{e1}), (\ref{c1}), (\ref{En}) is unique.
\end{proof}

\begin{remark}\label{rem3}
By Theorem~\ref{th5} in the case of maximal skew-symmetric operator $A$ a g.s. of (\ref{e1}), (\ref{c1}), (\ref{En}) is a semigroup g.s. If the operator $A$ is not maximal, this is not true anymore. Let us confirm this by the following simple example. Let $H=L^2([0,1])$, $Au=u'$, $u\in D(A)=\{ u=u(x)\in W_2^1([0,1]), \ u(1)=u(0)=0\}$. Obviously, $A$ is a closed skew symmetric operator with the deficiency indexes $d_\pm=1$. There are two skew-adjoint extensions of this operator
$A_1u=A_2u=u'$, defined in the domains $D(A_1)=\{ u=u(x)\in W_2^1([0,1]), \ u(1)=u(0)\}$, $D(A_2)=\{ u=u(x)\in W_2^1([0,1]), \ u(1)=-u(0)\}$. It is easy to realize that the corresponding orthogonal groups have the form
\[ e^{-tA_1}u(x)=u_p(x-t), \quad e^{-tA_2}u(x)=u_{ap}(x-t),
\]
where the functions $u_p(x)$, $u_{ap}(x)$ are, respectively, periodic ant anti-periodic extensions of the function $u(x)\in L^2([0,1])$ on the whole line, so that  $u_p(x+1)=u_p(x)$, $u_{ap}(x+1)=-u_{ap}(x)$ for all $x\in\R$, $u_p(x)=u_{ap}(x)=u(x)$ for $x\in [0,1]$. Let $u_0=u_0(x)\in H$, $u_0\not=0$. Then the functions $u_1(t)=e^{-tA_1}u_0$, $u_2(t)=e^{-tA_2}u_0$ are semigroup g.s. of (\ref{e1}), (\ref{c1}), which satisfy (\ref{En}) with the equality sign: $\|u_1(t)\|=\|u_2(t)\|=\|u_0\|$. The function \[u(t)=\frac{1}{2}(u_1(t)+u_2(t))=(u_{0p}(x-t)+u_{0ap}(x-t))/2\]
is a g.s. of the same problem and it satisfies inequality (\ref{En}) as well because
\[
\|u(t)\|\le\frac{1}{2}(\|u_1(t)\|+\|u_2(t)\|)=\|u_0\|.
\]
Evidently, the function $u(t)$ is $2$-periodic and $u(2k+1)=u(1)=(u_0(x)-u_0(x))/2=0$, $u(2k)=u(0)=u_0$ for all $k\in\N$.
Assuming that $u(t)=T_t u_0$ is a semigroup g.s., $T_t=e^{tB}$, we get the equality
$u_0=T_2u_0=T_1(T_1u_0)=T_1u(1)=T_10=0$ that contradicts to the assumption $u_0\not=0$. Hence, $u(t)$ is a g.s. of  (\ref{e1}), (\ref{c1}), (\ref{En}), which is not a semigroup g.s.
\end{remark}

\section{Some examples}

\subsection{A transport equation}\label{subsec1}
Let $a(x)=(a_1(x),\ldots,a_n(x)) \in L^2_{loc}(\R^n,\R^n)$ be a solenoidal vector in $\R^n$, that is,
\begin{equation}\label{sol}
\div a(x)=0 \ \mbox{ in } \D'(\R^n).
\end{equation}
We consider the transport equation (also called the continuity equation)
\begin{equation}\label{tre}
u_t+\div_x (au)=0,
\end{equation}
$u=u(t,x)$, $(t,x)\in\Pi=[0,+\infty)\times\R^n$. The notion of g.s. $u(t,x)\in L^2_{loc}([0,+\infty),L^2(\R^n))$ to the Cauchy problem for equation (\ref{tre}) with the initial condition
\begin{equation}\label{tri}
u(0,x)=u_0(x)\in L^2(\R^n)
\end{equation}
is defined by the standard integral identity (\ref{tr3}) below.

\begin{definition}\label{def2}
A function $u=u(t,x)\in L^\infty_{loc}([0,+\infty),L^2(\R^n))$ is called a g.s. of problem (\ref{tre}), (\ref{tri}) if for all $f=f(t,x)\in C_0^1(\Pi)$
\begin{equation}\label{tr3}
\int_{\Pi} u[f_t+a\cdot\nabla_x f]dtdx+\int_{\R^n} u_0(x)f(0,x)dx=0.
\end{equation}
In (\ref{tr3}) we use the notation $v\cdot w$ for the scalar product of vectors $v,w\in\R^n$.
\end{definition}

We introduce the Hilbert space $H=L^2(\R^n)$ and unbounded linear operator $A_0u=a(x)\cdot\nabla u(x)$ on $H$ with the domain  $D(A_0)=C_0^1(\R^n)\subset H$. This operator is skew-symmetric. In fact, for each $u,v\in C_0^1(\R^n)$
\begin{align*}(A_0u,v)_H+(u,A_0v)_H=\int_{\R^n} a(x)\cdot(v(x)\nabla u(x)+u(x)\nabla v(x))dx= \\ \int_{\R^n} a(x)\cdot\nabla (u(x)v(x))dx=0
\end{align*}
by solenoidality condition (\ref{sol}). Analyzing Definitions~\ref{def1},~\ref{def2}, we find that the notion of g.s. of the problem (\ref{tre}), (\ref{tri}) is consistent with the theory of g.s. to the abstract problem (\ref{e1}), (\ref{c1}).

We are going to demonstrate that in the case of Lipschitz coefficients $a_i(x)$, $i=1,\ldots,n$, the closure $A$ of the operator $A_0$ is a skew-adjoint operator. Thus, we assume that for some constant $m>0$
\begin{equation}\label{lip}
|a(x)-a(y)|\le m|x-y| \ \forall x,y\in\R^n.
\end{equation}
Here and everywhere below we use the notation $|z|$ for the
Euclidean norm of a finite-dimensional vector $z$ (including the modulus of a number).
By the Lipschitz condition (\ref{lip}) the generalized derivatives
$(a_i)_{x_j}(x)\in L^\infty(\R^n)$ for all $i,j=1,\ldots,n$.
In this case we can solve the problem (\ref{tre}), (\ref{tri}) by the classical method of characteristics.
Recall that characteristics of equation (\ref{tre}) are integral curves $(t,x(t))$ of the characteristic system of ODE
\begin{equation}\label{ch}
\dot x=a(x).
\end{equation}
In view of Lipschitz condition there exists a unique solution to a Cauchy problem for system (\ref{ch}). Besides, it also follows from the Lipschitz condition that the vector $a(x)$ may have at most linear growth at infinity:
\begin{equation}\label{subl}
|a(x)|\le c(1+|x|), \quad c=\const
\end{equation}
Therefore, solutions $x(t)$ of system (\ref{ch}) cannot reach infinity for finite time and therefore they are defined for all $t\in\R$. We denote by $x(t;t_0,x_0)$ the unique solution of system (\ref{ch}) satisfying the condition $x(t_0)=x_0$, and set $y(t_0,x_0)=x(0;t_0,x_0)$. From the relation $\frac{d}{dt} u(t,x(t))=(u_t+a\cdot\nabla_x u)(t,x(t))=0$ it follows the equality $u(t,x)=u_0(y(t,x))$. By the Liouville theorem the diffeomorphisms $x\to y(t,x)$ keep the Lebesgue measure in $\R^n$, which implies that the linear operators $T_t(u_0)=u_0(y(t,x))$ are orthogonal in $H=L^2(\R^n)$. It is also clear that these operators satisfy the group property $T_{t+s}=T_tT_s$, $t,s\in\R$, and the continuity condition $T_tu_0\to u_0$ in $H$ as $t\to 0$. Hence, $\{T_t\}_{t\in\R}$ is a $C_0$-group of orthogonal operators, and by the Stone theorem, (cf.  \cite[Chapter IX]{Ios},\cite[Theorem~4.7]{Kr}) $T_t=e^{tB}$, where the infinitesimal generator $B$ is a skew-adjoint operator. We will show that actually $B=-A$.

We will need some estimates of solutions to the stationary equation
\begin{equation}
\label{tras}
\phi+ha\cdot\nabla\phi=\psi
\end{equation}
with $h>0$. Let $M(x)=\|Da(x)\|$ be the operator norm of the $n\times n$ matrix $Da(x)=\left(\frac{\partial a_i(x)}{\partial x_j}\right)_{i,j=1}^n$. It readily follows from condition (\ref{lip}) that $M(x)\le m$. The following statements hold.

\begin{lemma}\label{lm1}
Let $\psi(x)\in C_0^1(\R^n)$ and $0<h<1/m$. Then there exists a solution $\phi(x)\in W_\infty^1(\R^n)$ of equation (\ref{tras}) such that
\begin{equation}\label{est}
|\phi(x)|\le c_1(1+|x|)^{-\alpha/h}, \quad |\nabla\phi(x)|\le c_2(1+|x|)^{-\alpha(1/h-m)},
\end{equation}
$c_1,c_2,\alpha$ are positive constants, moreover $\alpha=1/c$, where $c>0$ is the constant from condition (\ref{subl}).
\end{lemma}

\begin{proof}
On a characteristic $x=x(t)$ equation (\ref{tras}) turns to ODE  $\phi+h\phi'=\psi$ where $\phi=\phi(x(t))$, $\psi=\psi(x(t))$. A particular solution of this equation is given by the expression
\[
\phi(t)=\frac{1}{h}\int_{-\infty}^t e^{(s-t)/h}\psi(x(s))ds.
\]
Taking $x(s)=x(s;t,y)$, we arrive at the representation
\[
\phi(y)=\frac{1}{h}\int_{-\infty}^t e^{(s-t)/h}\psi(x(s;t,y))ds=\frac{1}{h}\int_{-\infty}^t e^{(s-t)/h}\psi(x(s-t;0,y))ds.
\]
Making the change $s-t\to s$, we obtain the equality
\begin{equation}\label{repr1}
\phi(y)=\frac{1}{h}\int_{-\infty}^0 e^{s/h}\psi(x(s;0,y))ds.
\end{equation}
Since the function $\psi(x)$ is bounded and continuous, the integral in (\ref{repr1}) converges uniformly with respect to  $y\in\R^n$, consequently, $\phi(y)$ is continuous. We choose such $r>0$ that the support $\supp\psi(x)$ lies in the ball $|x|\le r$. Let $|y|>r$ and $x(s)=x(s;0,y)$. Since $x'(s)=a(x(s))$ then, with the help of (\ref{subl}), we obtain that
\[
\frac{d}{ds}|x(s)|\le |a(x(s))|\le c(1+|x(s)|).
\]
Therefore, $\frac{d}{ds}\ln(1+|x(s)|)\le c$ and after integration on $[s,0]$ we obtain the inequality
\[
\ln\left(\frac{1+|y|}{1+|x(s)|}\right)\le c|s|=-cs.
\]
It follows from this inequality that $|x(s)|>r$ for $s>s(y)\doteq-\frac{1}{c}\ln((1+|y|)/(1+r))$ and therefore  $\psi(x(s))=0$. Then, by representation (\ref{repr1})
\begin{align*}
|\phi(y)|=\left|\frac{1}{h}\int_{-\infty}^{s(y)} e^{s/h}\psi(x(s;0,y))ds\right|\le \frac{\|\psi\|_\infty}{h}\int_{-\infty}^{s(y)} e^{s/h}ds=\\ \|\psi\|_\infty e^{s(y)/h}=c_1(1+|y|)^{-\alpha/h},
\end{align*}
where $\alpha=1/c$, $c_1=\|\psi\|_\infty(1+r)^{\alpha/h}$. The obtained inequality remains true for $|y|\le r$ as well, because for such $y$
\[
|\phi(y)|\le \frac{\|\psi\|_\infty}{h}\int_{-\infty}^0 e^{s/h}ds=\\ \|\psi\|_\infty\le c_1(1+|y|)^{-\alpha/h}.
\]
The first of estimates (\ref{est}) is proved.

Then we notice that for $y_1,y_2\in\R^n$, $s<0$
\begin{align*}
|x(s;0,y_2)-x(s;0,y_1)|\le |y_2-y_1|+\int_s^0 |x'(t;0,y_2)-x'(t;0,y_1)|dt= \\
|y_2-y_1|+\int_s^0 |a(x(t;0,y_2))-a(x(t;0,y_1))|dt \le \\ |y_2-y_1|+m\int_s^0 |x(t;0,y_2)-x(t;0,y_1)|dt.
\end{align*}
By Gr\"{o}nwall's lemma we derive from this relation the estimate $|x(s;0,y_2)-x(s;0,y_1)|\le |y_2-y_1|e^{m|s|}$.
We find that the map $y\to x(s;0,y)$ is Lipschitz with the constant $e^{m|s|}$. Therefore, its generalized derivatives are bounded and the operator norm of the Jacobian matrix
$X(s)=D_yx(s;0,y)$ is bounded by $e^{m|s|}$: $\|X(s)\|\le e^{m|s|}$. It follows from the equality  $\nabla_y\psi(x(s;0,y))=X(s)^\top\nabla_x \psi(x(s;0,y))$ (valid for almost each $(s,y)$) that
\begin{equation}\label{est1}
|\nabla_y\psi(x(s;0,y))|\le \|\nabla\psi\|_\infty e^{m|s|}.
\end{equation}
Representation (\ref{repr1}) implies that for a.e. $y\in\R^n$
\begin{equation}\label{repr2}
\nabla\phi(y)=\frac{1}{h}\int_{-\infty}^0 e^{s/h}\nabla_y\psi(x(s;0,y))ds.
\end{equation}
By estimate (\ref{est1}) for $0<h<1/m$ the integral in the right-hand side of this equality converges uniformly with respect to $y\in\R^n$. As we have already established, for $0>s>s(y)=-\frac{1}{c}\ln((1+|y|)/(1+r))$ (if $|y|\le r$ we set
$s(y)=0$) the function $\psi(x(s;0,y))$ equals zero, consequently its gradient $\nabla_y\psi(x(s;0,y))$ is zero as well. Therefore,
\[
\nabla\phi(y)=\frac{1}{h}\int_{-\infty}^{s(y)} e^{s/h}\nabla_y\psi(x(s;0,y))ds,
\]
which implies the estimate
\begin{align*}
|\nabla\phi(y)|\le\frac{1}{h}\int_{-\infty}^{s(y)} e^{s/h}|\nabla_y\psi(x(s;0,y))|ds\le
\|\nabla\psi\|_\infty h^{-1}\int_{-\infty}^{s(y)} e^{(1/h-m)s}ds= \\ \frac{\|\nabla \psi\|_\infty}{1-mh}e^{(1/h-m)s(y)}=c_2(1+|y|)^{-\alpha(1/h-m)},
\end{align*}
where, as before, $\alpha=1/c$, and $c_2=\frac{\|\nabla\psi\|_\infty}{1-mh}(1+r)^{\alpha(1/h-m)}$. Hence, we derive the second estimate in (\ref{est}) (after replacing $y$ by $x$).
The proof is complete
\end{proof}
We need also the fact that finite functions from the space $W_2^1(\R^n)$ lie in the domain of $A$.

\begin{lemma}\label{lm1a}
Let $u(x)\in W_2^1(\R^n)$ be a function with compact support. Then $u(x)\in D(A)$ and $(Au)(x)=a(x)\cdot\nabla u(x)$.
\end{lemma}

\begin{proof}
We take a sequence $u_k(x)\in C_0^1(\R^n)$ such that $u_k\to u$ as $k\to\infty$ in $W_2^1(\R^n)$ and that
supports $\supp u_k$ are contained in some common ball $|x|\le R$. For instance we can take the sequence of averaged functions $u_k(x)=k^n\int_{\R^n}u(x-y)\rho(ky)dy$ with a kernel $\rho(z)\in C_0^1(\R^n)$, $\rho(z)\ge 0$, $\int_{\R^n}\rho(z)dz=1$. Then $u_k\in D(A_0)$ and as $k\to\infty$
\[ u_k\to u, \quad A_0u_k=a(x)\cdot\nabla u_k(x)\to a(x)\cdot\nabla u(x) \ \mbox{ in } H,\]
where we take into account that $|a(x)|\le c(1+R)$ in the ball $|x|\le R$ by (\ref{subl}). Since $A$ is a closure of  the operator $A_0$ we conclude that $u\in D(A)$ and  $Au=a(x)\cdot\nabla u(x)$. The proof is complete.
\end{proof}

Now we are ready to prove that $A$ is a skew-adjoint operator.

\begin{proposition}\label{pro1}
The operator $A$ is skew-adjoint.
\end{proposition}

\begin{proof}
Let $h>0$ be so small that $\alpha(1/h-m)>1+n/2$ and $\psi(x)\in C_0^1(\R^n)$. By Lemma~\ref{lm1} there exists a function  $\phi(x)\in W_\infty^1(\R^n)$, which solves equation (\ref{tras}) and satisfies estimates (\ref{est}).
Let us show that $\phi(x)\in D(A)$ and that $\phi+hA\phi=\psi$.
First of all we remark that $\alpha/h>1+n/2$ and it follows from (\ref{est}) and (\ref{subl}) that $\phi(x), |a(x)|\phi(x)\in H=L^2(\R^n)$. We choose such a function $\rho(y)\in C_0^1(\R^n)$ that $\rho(y)\ge 0$ and $\rho(0)=1$. By Lemma~\ref{lm1a} for all $k\in\N$ the functions $u_k\doteq\phi(x)\rho(x/k)\in D(A)$ and $Au_k=a(x)\cdot\nabla u_k(x)$. Obviously, $u_k\to\phi$ in $H$ as $k\to\infty$.
Further, the functions
\begin{equation}\label{f1}
v_k\doteq Au_k=(a(x)\cdot\nabla\phi(x))\rho(x/k)+k^{-1}\phi(x)a(x)\nabla_y\rho(x/k).
\end{equation}
Since $a(x)\cdot\nabla\phi(x)=(\psi(x)-\phi(x))/h\in H$ then
\[
(a(x)\cdot\nabla\phi(x))\rho(x/k)\mathop{\to}_{k\to\infty} a(x)\cdot\nabla\phi(x)
\]
a.e. in $\R^n$, and therefore in $H$. Using that
\[
|\phi(x)a(x)\nabla_y\rho(x/k)|\le\|\nabla\rho\|_\infty|a(x)||\phi(x)|\in H,
\]
we find that second term in (\ref{f1}) converges to zero in $H$ as $k\to\infty$. Thus, we established that in the limit as $k\to\infty$
$u_k\to\phi$, $v_k\to a(x)\cdot\nabla\phi(x)$ in $H$. Since the operator $A$ is closed, we conclude that
$\phi\in D(A)$ and $A\phi(x)=a(x)\cdot\nabla\phi(x)=(\psi(x)-\phi(x))/h$. In particular, $\phi+hA\phi=\psi$ and therefore  $\psi\in\Im(E+hA)$. But $\psi(x)\in C_0^1(\R^n)$ is arbitrary and we see that $C_0^1(\R^n)\subset\Im(E+hA)$ (notice also that the constant $\alpha=1/c$ and the above choice of the parameter $h$ does not depend on $\psi$). Remind that the subspace $\Im(E+hA)$ is closed in $H$ and since $C_0^1(\R^n)$ is dense in $H$ we obtain that $\Im(E+hA)=H$. By skew-symmetricity of $A$ the property $\Im(E+hA)=H$ holds for all $h>0$. Similarly we prove that $\Im(E-hA)=H$ (for that we can just repeat our arguments replacing the field $a(x)$ by  $-a(x)$).
Thus, the deficiency indexes of the operator $A$ are both zero, that is, this operator is skew-adjoint. The proof is complete.
\end{proof}

According to Theorems~\ref{th2},~\ref{th4}, existence and uniqueness of g.s. to problem (\ref{tre}), (\ref{tri}) follow from Proposition~\ref{pro1}. Moreover, $-A$ is an infinitesimal generator of the semigroup $T_t(u_0)(x)=u_0(y(t,x))$ of g.s., an we obtain the announced equality $B=-A$. Uniqueness of g.s. to the problem (\ref{tre}), (\ref{tri}) (both forward and backward) was also established in the case when the coefficients have Sobolev regularity in the paper by R.J.~DiPerna and P.L.~Lions  \cite{DiL}, later these results were extended to more general case of BV coefficients, see paper \cite{Amb}. According to Corollary~\ref{cor1} in these cases the operator $A$ is skew-adjoint. But, in general case there are numerous examples of non-uniqueness of g.s. to problem (\ref{tre}), (\ref{tri}), see for instance \cite{Aiz,Br,CLR,Dep,PaTr}, so that the operator $A$ may fail to be skew-adjoint.

Notice that for a bounded vector of coefficients $a(x)$ the statement of Corollary~\ref{cor1}, applied to problem (\ref{tre}), (\ref{tri}), follows from results of \cite{BoCr}, see also \cite{ufa,prep}.

\medskip
We underline that the global Lipschitz condition is essential for the statement of Proposition~\ref{pro1}. In fact, let 
$n=2$, $a=a(x,y)=(x^2,-2xy)$ (notice that $\div a=0$, as required), $A$ be a closure of the operator $A_0u=x^2u_x-2xyu_y$. As is easy to verify, for each $v(z)\in C_0^\infty(\R)$ the functions
\begin{equation}\label{ex1}
u(x,y)=\left\{\begin{array}{lcr} v(yx^2)e^{\frac{1}{hx}} & , & hx<0, \\ 0 & , & hx\ge 0 \end{array}\right.
\end{equation}
are solutions of the resolvent equation $u-hA^*u=0$, $h\in\R\setminus\{0\}$. Moreover, 
$u(x,y)\in C^\infty(\R^2)\cap L^\infty(\R^2)\cap L^2(\R^2)$, $\|u\|_\infty=\|v\|_\infty$, $\|u\|_2=\|v\|_2\sqrt{|h|/2}$.
By the arbitrariness of $v(z)$, the operators $E-hA^*$ have infinite-dimensional kernels. Therefore, $\codim\Im(E-hA)=\dim\ker(E-hA^*)=\infty$ and the operator $A$ has the same infinite deficiency indexes. Hence, this operator is not skew-adjoint (but it admits a skew-adjoint extension). Notice that the functions (\ref{ex1}) do not belong to $D(A)$ despite their smoothness, otherwise, they satisfy the equality $u+hAu=0$, which contradicts to injectivity of the operators $E+hA$.

For the modified coefficients $a=a(x,y)=((x_+)^2,-2x_+y)$, with $x_+=\max(x,0)$, functions
(\ref{ex1}) lie in $\ker(E-hA^*)$ only for $h<0$. If $h>0$, this kernel is trivial. This means that the deficiency indexes $d_\pm$ of the operator $A$ are as follows: $d_-=0$, $d_+=\infty$. In particular, $A$ is a maximal skew-symmetric operator, which is not skew-adjoint.

\subsection{Linearized Euler system}\label{subsec2}
We consider the linearized Euler system with the same solenoidal vector of coefficients $a(x)\in L^2_{loc}(\R^n,\R^n)$ as in the previous section~\ref{subsec1}:
\begin{equation}\label{Es}
u_t+\sum_{j=1}^n (a_j(x)u)_{x_j}+\nabla p=0, \quad \div_x u=0,
\end{equation}
where $(t,x)\in\Pi$, $u=(u^1(t,x),\ldots,u^n(t,x))\in L^\infty_{loc}([0,+\infty),H)$, $H\subset L^2(\R^n,\R^n)$ is a Hilbert space of solenoidal vector fields in $\R^n$, $p=p(t,x)$ is a pressure. We study the Cauchy problem with the initial condition
\begin{equation}\label{Esi}
u(0,x)=u_0(x)\in H.
\end{equation}

\begin{definition}\label{def3}
A vector $u=u(t,x)\in L^\infty_{loc}([0,+\infty),H)$ is called a g.s. of problem (\ref{Es}), (\ref{Esi}) if for each test
vector-function $f=f(t,x)\in C_0^1(\Pi,\R^n)$ such that $\div_x f=0$
\[
\int_{\Pi} [u\cdot f_t+ u\cdot (a\cdot\nabla_x) f]dtdx+\int_{\R^n} u_0(x)\cdot f(0,x)dx=0.
\]
\end{definition}
The choice of solenoidal test vectors allows to exclude the pressure $p=p(t,x)$ from the system.
Definition~\ref{def3} is consistent with Definition~\ref{def1} if we define the operator $A_0$ by the equality $A_0 u=Pv$, where $u\in C_0^1(\R^n,\R^n)\cap H=D(A_0)$,
$\displaystyle v=v(x)=\sum_{j=1}^n a_j(x)u_{x_j}(x)$ while $P:L^2(\R^n,\R^n)\to H$ is an orthogonal projector of the space $L^2(\R^n,\R^n)$ onto its closed linear subspace $H$. If $u_1=(u_1^k)_{k=1}^n,u_2=(u_2^k)_{k=1}^n\in D(A_0)$ then
\begin{align*}
(A_0u_1,u_2)+(u_1,A_0u_2)=(Pv_1,u_2)+(u_1,Pv_2)=(v_1,u_2)+(u_1,v_2)=\\
\int_{\R^n}\sum_{j,k=1}^n a_j(x)((u_1^k)_{x_j}u_2^k+u_1^k(u_2^k)_{x_j})(x)dx=
\int_{\R^n}\sum_{j,k=1}^n a_j(x)(u_1^k u_2^k)_{x_j}dx=0
\end{align*}
by condition (\ref{sol}). Thus, the operator $A_0$ is skew-symmetric.

We are going to show that in the case of Lipschitz coefficients $a_j(x)$, $j=1,\ldots,n$, the operator $A_0$ admits a skew-adjoint extension. For that we need to restore the term $\nabla p$ from the equality
\[
\div\left(\sum_{j=1}^n (a_j u)_{x_j}+\nabla p\right)=0.
\]
Taking into account that $\div u=0$ for $u=(u^1,\ldots,u^n)\in H$, we obtain the relation
\[
\sum_{k,j=1}^n ((a_j)_{x_k}(x)u^k)_{x_j}+\Delta p=0,
\]
understood in the sense of distributions. Applying the Fourier transform, we arrive at the equality
\begin{equation}\label{ei1}
\tilde p(\xi)=\sum_{j=1}^n \F\left(\sum_{k=1}^n(a_j)_{x_k}u^k\right)(\xi)\frac{i\xi_j}{|\xi|^2}.
\end{equation}
Here $i$ is the imaginary unit, $\tilde u(\xi)=\F(u)(\xi)$ is the Fourier transform of a function (more generally, a distribution) $u(x)$, which is defined for $u\in L^1(\R^n)$ by the standard equality
\[
\tilde u(\xi)=\F(u)(\xi)=(2\pi)^{-n/2}\int_{\R^n}e^{-i\xi\cdot x}u(x)dx.
\]
Then we find $\F(\nabla p)(\xi)=i\xi\tilde p(\xi)$ and it follows from (\ref{ei1}) that
\begin{equation}\label{ei2}
\F(\nabla p)^l(\xi)=-\sum_{j=1}^n\F\left(\sum_{k=1}^n(a_j)_{x_k}u^k\right)(\xi)\frac{\xi_l\xi_j}{|\xi|^2}.
\end{equation}
Since the functions $(a_j)_{x_k}(x)$, $\xi_l\xi_j/|\xi|^2$ are bounded while the Fourier transform is an unitary linear operator in $L^2(\R^n)$, we conclude that $\nabla p=Tu$ is a bounded linear operator in $L^2(\R^n,\R^n)$. By the construction $A_0u=\sum_{j=1}^n (a_j(x)u)_{x_j}+Tu$ for $u(x)\in D(A_0)\subset H$. Let us define the operator $B_0$ in
$L^2(\R^n,\R^n)$ according to the equalities $(B_0u)^l=\sum_{j=1}^n a_j(x)u^l_{x_j}$, $l=1,\ldots,n$, with the domain $D(B_0)=C_0^1(\R^n,\R^n)$. The operator $B_0$ is a direct sum of $n$ scalar transport operators considered in the previous section. By Proposition~\ref{pro1} the closure $B$ of operator $B_0$ is a skew-adjoint operator. The operator $\tilde A_0=B_0+T$ is defined on $D(B_0)$ and it is an extension of the operator $A_0$. Since $T$ is a bounded operator, the closure $\tilde A$ of the operator $\tilde A_0$ coincides with $B+T$.
Recall that the operator $B$ is skew-adjoint, therefore $\Im(E-hB)=L^2(\R^n,\R^n)$ and $\|(E-hB)^{-1}\|\le 1$ for all $h\in\R$. Since $\tilde A=B+T$ is a bounded perturbation of $B$, then the property $\Im(E-h\tilde A)=L^2(\R^n,\R^n)$ remains valid if $|h|$ is sufficiency small.

Our next step is to demonstrate that the space $H$ is invariant for the resolvent $(E-h\tilde A)^{-1}$, where $|h|<h_0$ is small enough.
Without loss of generality, we may suppose that $h>0$. The case of a negative $h$ is treated using the change
$a(x)$ by $-a(x)$.

\begin{lemma}\label{lm2}
Let $u\in D(\tilde A)$, $u-h\tilde Au=v\in H$, and $h>0$ be small enough. Then $u\in H$, i.e., $\div u=0$ in $\D'(\R^n)$.
\end{lemma}

\begin{proof}
Let $\psi(x)\in C_0^1(\R^n)$, and $\phi(x)\in W_\infty^1(\R^n)$ be the function constructed in Lemma~\ref{lm1}. We chose $h>0$ so small that $\alpha(1/h-m)>1+n/2$. Then it follows from estimates (\ref{subl}), (\ref{est}) that $|\nabla\phi(x)|\in L^2(\R^n)$, $|a(x)||\nabla\phi(x)|\in L^2(\R^n)$. Multiplying the equality $u-h\tilde Au=v$ by the potential vector $\nabla\phi$, we obtain the equality
\[(u,\nabla\phi)-h(\tilde Au,\nabla\phi)=(v,\nabla\phi)=0.\]
Remark that, by Parseval's identity,
\begin{align*}
(Tu,\nabla\phi)=(\F(Tu),\F(\nabla\phi))=\\ -\int \sum_{j=1}^n\F\left(\sum_{k=1}^n(a_j)_{x_k}u^k\right)(\xi)
\sum_{l=1}^n\frac{\xi_l\xi_j}{|\xi|^2}\overline{i\xi_l\F(\phi)(\xi)}d\xi=\\
-\int_{\R^n}\sum_{j=1}^n\F\left(\sum_{k=1}^n(a_j)_{x_k}u^k\right)(\xi)\overline{i\xi_j\F(\phi)(\xi)}d\xi=
-\int_{\R^n}\sum_{j,k=1}^n(a_j)_{x_k}(x)u^k\phi_{x_j}(x)dx.
\end{align*}
Then, for $u\in D(\tilde A_0)=D(B_0)$
\begin{align*}
(Bu,\nabla\phi)=\int_{\R^n}\sum_{j,k=1}^n a_j(x)u^k_{x_j}(x)\phi_{x_k}(x)dx=\\
\int_{\R^n}\sum_{j,k=1}^n (a_j(x)u^k)_{x_j}(x)\phi_{x_k}(x)dx=\int_{\R^n}\sum_{j,k=1}^n (a_j(x)u^k)_{x_k}(x)\phi_{x_j}(x)dx.
\end{align*}
By the above relations we find the equality
\begin{align}\label{ei3}
(\tilde Au,\nabla\phi)=(Bu,\nabla\phi)+(Tu,\nabla\phi)=\nonumber\\
\int_{\R^n}\sum_{j,k=1}^n a_j(x)(u^k)_{x_k}(x)\phi_{x_j}(x)dx=
\int_{\R^n}\div u(x)\sum_{j=1}^n a_j(x)\phi_{x_j}(x)dx.
\end{align}
Thus, for all $u\in D(\tilde A_0)$
\begin{align*}
(u-h\tilde Au,\nabla\phi)=\int_{\R^n}u(x)\cdot\nabla\phi(x)dx-h(\tilde Au,\nabla\phi)=\\ -\int_{\R^n}\div u(x)(\phi+h\sum_{j=1}^n a_j(x)\phi_{x_j}(x))dx=\\ -\int_{\R^n}\div u(x)\psi(x)dx=\int_{\R^n}u(x)\cdot\nabla\psi(x)dx=(u,\nabla\psi).
\end{align*}
Since $D(\tilde A_0)$ is dense in $D(\tilde A)$ in the graph norm of $\tilde A$ then the relation
\[(u-h\tilde Au,\nabla\phi)=(u,\nabla\psi)\] can be extended by continuity to the case $u\in D(\tilde A)$. By the assumption $u-h\tilde Au=v\in H$ and therefore
$(u-h\tilde Au,\nabla\phi)=0$. Then $(u,\nabla\psi)=0$. But $\psi(x)\in C_0^1(\R^n)$ is arbitrary, and we conclude that $\div u(x)=0$ in $\D'(\R^n)$, that is, $u\in H$. The proof is complete.
\end{proof}

It follows from (\ref{ei3}) that for $u\in D(A_0)=D(B_0)$, $\phi\in C_0^2(\R^n)$
\begin{align}\label{ei4} (\tilde Au,\nabla\phi)=\int_{\R^n}\div u(x)\sum_{j=1}^n a_j(x)\phi_{x_j}(x)dx=\nonumber\\
-\int_{\R^n}u(x)\cdot\nabla\left(\sum_{j=1}^n a_j(x)\phi_{x_j}(x)\right)dx
\end{align}
and since $D(\tilde A_0)$ is dense in $D(\tilde A)$ in the graph norm of $\tilde A$ then equality (\ref{ei4}) remains fulfilled for $u\in D(\tilde A)$. If $u\in H\cap D(\tilde A)$ then $\div u(x)=0$ and it follows from (\ref{ei4}) that  $(\tilde Au,\nabla\phi)=0$ for all $\phi\in C_0^2(\R^n)$, that is, $\div\tilde Au=0$ in $\D'(\R^n)$. Hence, the subspace $H$ is invariant for $\tilde A$. By the construction, vectors $Tu$ are potential, this implies that $Tu\perp H$ and $PTu=0$. Therefore, $\tilde Au=PBu+PTu=PBu$ for $u\in D(\tilde A)\cap H$. This and skew-symmetricity of the operator $B$ easily imply that the restriction $\tilde A|_H$ is a skew-symmetric operator in $H$.
Since $\Im(E-h\tilde A)=L^2(\R^n,\R^n)$ for sufficiently small $h>0$ then for each $v\in H$ there exists such $u\in D(\tilde A)$ that $u-h\tilde Au=v$. By Lemma~\ref{lm2} we find that $u\in H\cap D(\tilde A)$. Hence,
$\Im (E-h\tilde A|_H)=H$. The similar statement holds after replacement $h$ by $-h$.
Thus, the skew-symmetric operator $\tilde A|_H$ has null deficiency indexes and therefore it is skew-adjoint.
Since $\tilde A_0|_H=A_0$ then $\tilde A|_H$ is a skew-adjoint extension of the operator $A$ because it is the closure of $A_0$. We have proven the following result.

\begin{proposition}\label{pro2}
The operator $A$ admits a skew-adjoint extension. In particular, deficiency indexes $d_+$ and $d_-$ of $A$ coincide.
\end{proposition}

To prove skew-adjointness of the operator $A$ itself, we need a variant of the famous DiPerna-Lions commutation lemma
\cite[Lemma~II.1]{DiL}. Suppose $\rho(z)\in C_0^1(\R^n)$, $\supp\rho\subset B_1(0)=\{z\in\R^n | |z|\le 1\}$, $\rho(z)\ge 0$, $\int_{\R^n}\rho(z)dz=1$. For $u(x)\in L^1_{loc}(\R^n)$ we introduce the corresponding averaged functions
\[u_k(x)=u*\rho_k(x)=\int_{\R^n}\rho_k(x-y)u(y)dy,\]
where $k\in\N$, $\rho_k(z)=k^n\rho(kz)$.

\begin{lemma}\label{lmDL}
Let $a(x)\in C(\R^n)$ satisfy the global Lipschitz condition (\ref{lip}), that is
$|a(x)-a(y)|\le m|x-y|$ for all $x,y\in\R^n$; $u(x)\in L^p(\R^n)$, $1\le p<\infty$. Then
\begin{equation}\label{DL}
\frac{\partial}{\partial x_j}(au_k-(au)_k)(x)\to 0 \ \mbox{ in } L^p(\R^n)
\end{equation}
for all $j=1,\ldots,n$.
\end{lemma}

\begin{proof}
Since \[(au_k-(au)_k)(x)=k^n\int_{\R^n}(a(x)-a(y))\rho(k(x-y)) u(y)dy,\]
there exists the generalized derivatives
\begin{align}\label{rel}
\frac{\partial}{\partial x_j}(au_k-(au)_k)(x)=k^n\int_{\R^n}a_{x_j}(x)\rho(k(x-y)) u(y)dy+\nonumber\\ k^{n+1}\int_{\R^n}(a(x)-a(y))\rho_{z_j}(k(x-y)) u(y)dy=I_{1k}(x)+I_{2k}(x).
\end{align}
The first term in this sum
\begin{equation}\label{rel1}
I_{1k}(x)=a_{x_j}(x)\int_{\R^n}\rho_k(x-y) u(y)dy=a_{x_j}(x)u_k(x)\to a_{x_j}(x)u(x)
\end{equation}
as $k\to\infty$ in $L^p(\R^n)$ because $u_k\mathop{\to}\limits_{k\to\infty} u$ in $L^p(\R^n)$ by the known property of averaged functions while the derivative $a_{x_j}(x)\in L^\infty(\R^n)$ in view of the Lipschitz condition.
Next, we estimate the term $I_{2k}(x)$. Obviously,
\begin{align}\label{es1}
|I_{2k}(x)|\le k^{n+1}\int_{\R^n}|a(x)-a(y)||\rho_{z_j}(k(x-y))| |u(y)|dy\le \nonumber\\ \omega_k(x)\doteq mk^n\int_{\R^n}k|x-y||\rho_{z_j}(k(x-y))| |u(y)|dy.
\end{align}
By the properties of averaged functions $\omega_k(x)\in L^p(\R^n)$ and
\begin{equation}\label{rel2}
\omega_k(x)\to C|u(x)|
\end{equation}
as $k\to\infty$ both in $L^p(\R^n)$ and a.e. in $\R^n$. Here $\displaystyle C=m\int_{\R^n} |z||\rho_{z_j}(z)|dz=\const$.
Now, let $x$ be a common Lebesgue point of the vector $\nabla a(y)$ and the function $u(y)$. Then
\begin{align}\label{rel2-1}
I_{2k}(x)=k^{n+1}\int_{\R^n}(a(x)-a(y))\rho_{z_j}(k(x-y))(u(y)-u(x))dy+\nonumber\\ u(x)k^{n+1}\int_{\R^n}(a(x)-a(y))\rho_{z_j}(k(x-y))dy.
\end{align}
The first term in the right-hand part of (\ref{rel2-1}) is estimated as
\begin{align}\label{est2-1}
k^{n+1}\left|\int_{\R^n}(a(x)-a(y))\rho_{z_j}(k(x-y))(u(y)-u(x))dy\right|\le \nonumber\\
mk^n\int_{\R^n}k|x-y||\rho_{z_j}(k(x-y))||u(y)-u(x)|dy\mathop{\to}_{k\to\infty} 0
\end{align}
since $x$ is a Lebesgue point of $u(y)$. We introduce the function ${\displaystyle J_k(x)=k^{n+1}\int_{\R^n}(a(x)-a(y))\rho_{z_j}(k(x-y))dy}$ and represent it in the form
\begin{align}\label{rel2-2}
J_k(x)=k^{n+1}\int_{\R^n}(a(x)-a(y)-\nabla a(x)\cdot(x-y))\rho_{z_j}(k(x-y))dy+\nonumber\\ k^{n+1}\int_{\R^n}\nabla a(x)\cdot (x-y)\rho_{z_j}(k(x-y))dy=\nonumber\\
k^{n+1}\int_{\R^n}\int_0^1 (\nabla a(x+s(y-x))-\nabla a(x))\cdot(x-y)\rho_{z_j}(k(x-y))dsdy+\nonumber\\ \nabla a(x)\cdot\int_{\R^n}z\rho_{z_j}(z)dz.
\end{align}
We utilize here the identity
\[
a(x)-a(y)=\int_0^1\nabla a(x+s(y-x))\cdot(x-y)ds,
\]
which holds for a.e. $y\in\R^n$.
To estimate the first term in the right-hand part of (\ref{rel2-2}), we make the change of variables $z=s(x-y)$, resulting in
\begin{align}\label{est2-2}
k^{n+1}\left|\int_{\R^n}\int_0^1 (\nabla a(x+s(y-x))-\nabla a(x))\cdot(x-y)\rho_{z_j}(k(x-y))dsdy\right|\le\nonumber\\
k^n\int_{\R^n} |\nabla a(x-z)-\nabla a(x)|\rho_1(kz)dz,
\end{align}
where we denote
\[
\rho_1(y)=|y|\int_0^1s^{-n-1}|\rho_{z_j}(y/s)|ds.
\]
Remark that $\rho_1\ge 0$, $\supp\rho_1\subset B_1(0)$, $\displaystyle\int_{\R^n}\rho_1(y)dy=\int_{\R^n}|z||\rho_{z_j}(z)|dz$, and since $x$ is a Lebesgue point of the vector $\nabla a(y)$, we find
\[k^n\int_{\R^n} |\nabla a(x-z)-\nabla a(x)|\rho_1(kz)dz\mathop{\to}_{k\to\infty} 0\].
It follows from (\ref{rel2-2}) and (\ref{est2-2}) that
\begin{equation} \label{rel3}
J_k(x)\mathop{\to}_{k\to\infty} \nabla a(x)\cdot\int_{\R^n}z\rho_{z_j}(z)dz=-a_{x_j}(x).
\end{equation}
Here, we apply the integration by parts formula
\[\int_{\R^n}z\rho_{z_j}(z)dz=-\int_{\R^n}\frac{\partial z}{\partial z_j}\rho(z)dz=-e_j, \]
where $e_j$ is the $j$th basis vector in $\R^n$.
In view of (\ref{rel2-1}), (\ref{est2-1}) it follows from (\ref{rel3}) that $I_{2k}(x)\to -a_{x_j}(x)u(x)$ as $k\to\infty$ a.e. in $\R^n$. Further, we notice that
\[
|I_{2k}(x)+a_{x_j}(x)u(x)|^p\le 2^{p-1}(|I_{2k}(x)|^p+|a_{x_j}(x)u(x)|^p)\le 2^{p-1}((\omega_k(x))^p+|a_{x_j}(x)u(x)|^p).\]
The left-hand side of this inequality converges to zero a.e. in $\R^n$, while its right-hand side converges both in  $L^1(\R^n)$ and a.e. in $\R^n$, by relation (\ref{rel2}). Applying Fatou's lemma to the sequences
\[
2^{p-1}((\omega_k(x))^p+|a_{x_j}(x)u(x)|^p)\pm |I_{2k}(x)+a_{x_j}(x)u(x)|^p,
\]
we derive that
\[
\lim_{k\to\infty} \int_{\R^n}|I_{2k}(x)+a_{x_j}(x)u(x)|^pdx=0,
\]
that is, $\displaystyle I_{2k}(x)\mathop{\to}_{k\to\infty} -a_{x_j}(x)u(x)$ in $L^p(\R^n)$. This, together with  (\ref{rel}), (\ref{rel1}), gives the required relation (\ref{DL}). The proof is complete.
\end{proof}

Now we are ready to prove that the operator $A$ is skew-adjoint provided ${a(x)\in W_\infty^1(\R^n,\R^n)}$ (that is, in addition to the requirements of Proposition~\ref{pro2} we assume that the coefficients $a(x)$ are bounded).

\begin{proposition}\label{pro3}
Let $a_j(x)\in W_\infty^1(\R^n)$ for all $j=1,\ldots,n$. Then the operator $A$ is skew-adjoint.
\end{proposition}

\begin{proof}
Let us firstly show that $W_2^1(\R^n,\R^n)\cap H\subset D(A)$.
It is known, see for instance \cite[Lemma 3.1]{LadSol}, that the space $D(A_0)=C_0^1(\R^n,\R^n)\cap H$ is dense in  $W_2^1(\R^n,\R^n)\cap H$ (the proof is given for physical case $n=2,3$ but can be easily extended for arbitrary dimensions). Therefore, we can find a sequence $u_r\in D(A_0)$, $r\in\N$ such that the sequences $u_r$, $(u_r)_{x_j}$, $j=1,\ldots,n$, converge as $r\to\infty$ to $u(x)$ and, respectively, to $u_{x_j}$ in $L^2(\R^n,\R^n)$. By our assumption $a_j(x)\in L^\infty(\R^n)$, and we conclude that \[A(u_m)=\sum_{j=1}^n a_j(x)(u_m)_{x_j}(x)+Tu_m\to Bu+Tu=\tilde Au\] in $L^2(\R^n,\R^n)$. Since the operator $A$ is closed, we find that $u\in D(A)$ and that $Au=\tilde Au$.

Now, we assume that $u\in D(\tilde A)\cap H=D(B)\cap H$ and introduce the sequence of averaged vectors \[u_k(x)=k^n\int_{\R^n}\rho(ky)u(x-y)dy=k^n\int_{\R^n}\rho(k(x-y))u(y)dy.\]
It is clear that $u_k(x)\in W_2^1(\R^n,\R^n)$ and that $\displaystyle\div u_k(x)=\rho_k*\div u=0$. Hence, $u_k\in  W_2^1(\R^n,\R^n)\cap H\subset D(A)$.
In the limit as $k\to\infty$ the following relations hold.
\begin{equation}\label{cl1}
u_k\to u, \quad (Bu)_k=k^n\int_{\R^n}\rho(k(x-y))Bu(y)dy\to (Bu)(x) \ \mbox{ in } L^2(\R^n,\R^n).
\end{equation}
Notice that by solenoidality of the coefficients $a(x)$ and by commutativity of the averaging and differentiation operators
\[ Bu_k(x)=\sum_{j=1}^n (a_j(x)u_k(x))_{x_j}, \quad (Bu)_k(x)=\sum_{j=1}^n ((a_ju)_k)_{x_j}(x).\]
Then, by Lemma~\ref{lmDL} (with $p=2$)
\[Bu_k-(Bu)_k=\sum_{j=1}^n (a_j(x)u_k(x)-(a_ju)_k(x))_{x_j}\to 0 \ \mbox{ in } L^2(\R^n,\R^n)\]
and the second relation in (\ref{cl1}) can be written in the form
$Bu_k\to Bu$ in $L^2(\R^n,\R^n)$. Since the operator $T$ in $L^2(\R^n,\R^n)$ is bounded, we obtain the relation
\[
u_k\mathop{\to}_{k\to\infty} u, \quad Au_k=Bu_k+Tu_k\mathop{\to}_{k\to\infty} Bu+Tu=\tilde Au \ \mbox{ in } H.\]
Using again that the operator $A$ is closed, we conclude $u\in D(A)$ and $Au=\tilde Au$. So, we established that
$D(\tilde A)\cap H\subset D(A)$. We remind that the operator $\tilde A|_H$ is a skew-adjoint extension of $A$ and in particular $D(A)\subset D(\tilde A)\cap H$. Hence, $D(\tilde A)\cap H=D(A)$ and the operator $A=\tilde A|_H$ is skew-adjoint, as was to be proved.
\end{proof}

\end{document}